\documentclass[11pt, final]{amsart}
\usepackage{amsmath, amsxtra, amssymb, mathdots}
\usepackage{mathpazo}
\usepackage{verbatim}
\usepackage{multirow}

\usepackage[margin=1.3in]{geometry}
\usepackage{graphicx}
\usepackage[cmtip,all]{xy}

\usepackage[notcite, notref]{showkeys}
\usepackage[colorlinks, linkcolor=blue, citecolor=red, urlcolor=black]{hyperref}

\theoremstyle{plain}
\newtheorem*{main-theorem}{Main Theorem}
\newtheorem{theorem}[equation]{Theorem}
\newtheorem{prop}[equation]{Proposition}

\newtheorem{claim}[equation]{Claim}
\newtheorem*{claim*}{Claim}

\newtheorem{lemma}[equation]{Lemma}

\theoremstyle{definition}
\newtheorem{definition}[equation]{Definition}

\newtheorem{remark}[equation]{Remark}

\newtheorem{question}[equation]{Question}

\numberwithin{equation}{subsection}

\DeclareMathOperator{\Sym}{Sym}
\DeclareMathOperator{\proj}{Proj}
\DeclareMathOperator{\Stab}{Stab}

\DeclareMathOperator{\lspan}{span}

\DeclareMathOperator{\init}{in}

\DeclareMathOperator{\SL}{SL}
\DeclareMathOperator{\GL}{GL}

\DeclareMathOperator{\Grass}{Grass}

\DeclareMathOperator{\Res}{Res}
\DeclareMathOperator{\Hom}{Hom}

\newcommand{\gitq}{/\hspace{-0.25pc}/}

\def\co{\colon\thinspace} 
\def\QQ{\mathbb{Q}}

\def\bP{\mathbf{P}}

\def\CC{\mathbb{C}}

\begin{document}

\title{GIT semistability of Hilbert points of Milnor algebras} 

\begin{abstract} We study GIT semistability of Hilbert points of Milnor algebras of homogeneous forms. 
Our first result is that a homogeneous form $F$ in $n$ variables is GIT semistable with respect to the natural
$\SL(n)$-action if and only if the gradient point of $F$, which is the first non-trivial Hilbert point of the Milnor algebra of $F$, is semistable.
We also prove that the induced morphism on the GIT quotients is finite, and injective on the locus of stable forms.
Our second result is that the associated form of $F$, also known as the Macaulay inverse system of the Milnor
algebra of $F$, and which is apolar to 
the last non-trivial Hilbert point of the Milnor algebra,
is GIT semistable whenever $F$ is a smooth form. 
These two results answer questions of Alper and Isaev. 
\end{abstract}
\author{Maksym Fedorchuk}
\address{Department of Mathematics, Boston College, 140 Commonwealth Ave, Chestnut Hill, MA 02467, USA}
\email{maksym.fedorchuk@bc.edu}
\maketitle

\setcounter{tocdepth}{1}
\tableofcontents

\section{Introduction} 

The main results of this paper are the following two theorems:
\begin{theorem}\label{T:main} Let $F=F(x_1,\dots,x_n)$ be a homogeneous polynomial of degree $d+1$ in $n$ variables. 
Then $F$ is semistable with respect to the $\SL(n)$-action on $\CC[x_1,\dots,x_n]_{d+1}$ if and only if 
\[
\nabla (F):=\lspan\langle \partial F/\partial x_1,\dots,\partial F/\partial x_n\rangle \in \Grass\bigl(n, \CC[x_1,\dots,x_n]_{d}\bigr)
\]
is semistable with respect to the $\SL(n)$-action on the Grassmannian.
Furthermore, if $F$ is stable, then $\nabla(F)$ is polystable, and stable
if and only if $F$ is not a non-trivial sum of two polynomials in disjoint sets of variables. 
\end{theorem}
This theorem answers in affirmative the semistability part of \cite[Question 3.3]{alper-isaev-binary}.
In Proposition \ref{P:almost-embedding}, we apply this result to deduce that the morphism 
on GIT quotients induced by $\nabla$ is finite and generically injective, giving a partial answer
to the rest of \cite[Question 3.3]{alper-isaev-binary}.

The second result deals with \emph{the associated forms} of homogeneous regular sequences,
as defined by Alper and Isaev \cite{alper-isaev-assoc,alper-isaev-binary}.
We refer the reader to Subsection \ref{S:complete} for the definitions. 
\begin{theorem}\label{T:main-2}
Suppose that $g_1,\dots,g_n$ is a regular sequence of homogeneous degree $d$ polynomials
in $\CC[x_1,\dots,x_n]$. Then 
the associated form 
\[
\mathbf A (g_1,\dots,g_n) \in \bP\bigl(\CC[\partial/\partial x_1, \dots, \partial/\partial x_n]_{n(d-1)}\bigr)
\]
is semistable with respect to the $\SL(n)$-action.
In particular, for every smooth form $F$, 
the associated form of the regular sequence $\partial F/\partial x_1,\dots,\partial F/\partial x_n$ is semistable.
\end{theorem}
This answers in affirmative the semistability part of \cite[Question 3.1]{alper-isaev-binary}.

\subsubsection*{Notation} We work over the complex numbers, characteristic $0$ hypothesis being essential.
Given a vector space $V$, we let $\bP V =  \proj \bigl(\Sym V^{\vee}\bigr)$ 
denote the space of \emph{lines} in $V$. We denote by $\Grass\bigl(k, V\bigr)$
the Grassmannian of $k$-dimensional \emph{subspaces} of $V$. 
We generally keep the notation of \cite{alper-isaev-binary}, with the following
exceptions: (i) our $V$ is the dual of their $V$, and 
(ii) our default degree of a homogeneous form is $d+1$, and not $d$. 

\subsection{Stability of homogeneous forms}
\label{S:GIT}
Let $V$ be a $\CC$-vector space of dimension $n$. 
The Geometric Invariant Theory (GIT) gives a projective moduli space for the isomorphism classes of
degree $d+1$ hypersurfaces in $\bP V^{\vee}$ \cite[Chapter 4.2]{GIT}. 
Recall that a homogeneous form $F$ of degree $d+1$ on $V^{\vee}$ is \emph{semistable} 
with respect to the standard $\SL(V)$-action on $\Sym^{d+1} V$
if and only if the closure of the $\SL(V)$-orbit of $F$ in $\Sym^{d+1} V$ does not contain $\mathbf{0}$. 
We call a hypersurface semistable if it is defined by a semistable form. 
The GIT quotient $\mathbf{P} \bigl(\Sym^{d+1} V\bigr)^{ss} \gitq \SL(V)$ parameterizes orbits of semistable hypersurfaces 
that are closed in the semistable locus. 
Concretely, this moduli space is given by the projective spectrum of the graded algebra of $\SL(V)$-invariant
forms on $\Sym^{d+1} V$:
\begin{equation}
\mathbf{P} \bigl(\Sym^{d+1} V\bigr)^{ss} \gitq \SL(V) = \proj \bigoplus_{m\geq 0} \left(\Sym^{m} \bigl(\Sym^{d+1} V\bigr)^\vee \right)^{\SL(V)}.
\end{equation}
The Hilbert-Mumford Numerical Criterion \cite[Theorem 2.1]{GIT} gives, in principle, 
a way to determine the set of semistable hypersurfaces,
but a complete geometric description of the resulting quotient is available only for certain small values of $n$ or $d$.

Since the $\SL(V)$-action on $\bP\bigl(\Sym^{d+1} V\bigr)$ admits a unique (up to scaling) linearization,
there exists only one notion of $\SL(V)$-stability for elements of $\bP\bigl(\Sym^{d+1} V\bigr)$. 
However, by considering the Milnor algebra associated to a homogeneous form
and the Hilbert points of this algebra, one obtains \emph{a priori} different variants of stability. 
The goal of this paper is to explore two such variants appearing in the work of Alper and Isaev 
\cite{alper-isaev-assoc,alper-isaev-binary} on an invariant-theoretic approach
to the reconstruction problem arising from the well-known Mather-Yau theorem.

\subsection{Complete intersection algebras and associated forms}
\label{S:complete}
This section introduces some background for the main results and restates 
the GIT problems introduced in \cite{alper-isaev-binary} 
in the language of Hilbert points. 

Let $V$ be a vector space of dimension $n$. Denote $S=\Sym V$. 
For a point $[W] \in \Grass\bigl(n, \Sym^{d} V\bigr)$, that is, 
an $n$-dimensional linear space $W$ of degree $d$ homogeneous forms on $V^{\vee}$, 
we form the ideal $I_W:=(W) \subset S$
and the quotient algebra $S_W:=S/I_W.$

\begin{definition}
The  \textbf{$m^{th}$ Hilbert point of $S_W$} is the short exact sequence
\[
0 \to (I_W)_{m} \to \Sym^m V \to (S_W)_{m} \to 0,
\]
that we regard as a point of the Grassmannian $\Grass\bigl(\dim (I_W)_{m}, \Sym^m V\bigr)$.
\end{definition}

GIT stability of $m^{th}$ Hilbert points of the homogeneous coordinate rings of projective schemes 
(especially, in dimensions $0$, $1$, and $2$, and for $m\gg 0$) 
is a classical subject in moduli theory. Although $\proj S_W$ is empty  
for a generic choice of $W$, the problem of GIT stability of $m^{th}$ Hilbert points of $S_W$ is still interesting, but only for a 
finite range of $m$. 

Suppose $W=\lspan\langle g_1,\dots, g_n\rangle$, where $g_i$'s are linearly independent degree
$d$ forms on $V^{\vee}$.
Then $g_1,\dots, g_n$
form a regular sequence in $S$ if and only if $\dim S_W=0$ (Krull's Hauptidealsatz) 
if and only if the locus given by
\[
g_1= \cdots= g_n=0
\]
is empty in $\bP V^{\vee}$ (Hilbert's Nullstellensatz)
if and only if $S_W$ is Artinian. Moreover, if any of the above equivalent conditions hold, 
then $S_W$ is a graded local Artinian Gorenstein $\CC$-algebra with socle in degree $n(d-1)$. 
The formula for the socle degree can be obtained by using adjunction to compute the dualizing module of $S_W$:
\[
\omega_{S_W}\simeq \omega_{S}\bigl(nd\bigr)\otimes S_W \simeq S\bigl(-n+nd\bigr)\otimes S_W=S_W\bigl(n(d-1)\bigr),
\]
or by noting (cf. \cite[Theorem 11.1]{AM}) that the Hilbert function of $S_W$ is 
\[
\dfrac{(1-t^d)^n}{(1-t)^n}=(1+t+\cdots+t^{d-1})^n.
\]

Recall that the locus given by
$g_1= \cdots= g_n=0$ is empty in $\bP V^{\vee}$ if and only if the resultant of $g_1,\dots, g_n$ 
is zero \cite[Chapter 13]{GKZ}. It
follows that there exists an $\SL(V)$-invariant divisor
\[
\Res \subset \Grass\bigl(n, \Sym^{d} V\bigr)
\] parameterizing subspaces that do not generate a complete
intersection ideal. We denote by $\Grass\bigl(n, \Sym^{d} V\bigr)_{\Res}$ the affine complement of $\Res$.

Let $\iota(m)=\dim (I_W)_{m}$, where $[W]\in \Grass\bigl(n, \Sym^{d} V\bigr)_{\Res}$. Note
that $\iota(m)$ is simply the coefficient of $t^m$ in the Hilbert function $(1+t+\cdots+t^{d-1})^n$ of $S_W$.
It follows from the above discussion that for each $d \leq m \leq n(d-1)$, there is a rational map
\[
H_m\co \Grass\bigl(n, \Sym^{d} V\bigr) \dashrightarrow \Grass\bigl(\iota(m), \Sym^m V\bigr)
\]
assigning to $[W]$ the $m^{th}$ Hilbert point of $S_W$. By construction, this map is a morphism on 
$\Grass\bigl(n, \Sym^{d} V\bigr)_{\Res}$. Moreover, this morphism 
is equivariant with respect to the natural actions of $\SL(V)$ on both sides. 

Following \cite{alper-isaev-binary},
we also denote $H_{n(d-1)}$ by $\mathbf A$. 
For $[W]\in \Grass\bigl(n, \Sym^{d} V\bigr)_{\Res}$, we have
\[
\mathbf A(W)=\left[\Sym^{n(d-1)} V \to (S_W)_{n(d-1)}\to 0\right]\in \bP\left(\bigl(\Sym^{n(d-1)} V\bigr)^{\vee}\right),
\]
where we have used $\dim (S_W)_{n(d-1)}=1$ to identify $\mathbf A(W)$ with a point in the
space of lines in $\bigl(\Sym^{n(d-1)} V\bigr)^{\vee}$.
Using the natural isomorphism 
\[
\bP\left(\bigl(\Sym^{n(d-1)} V\bigr)^{\vee}\right)\simeq \bP(\Sym^{n(d-1)} V^{\vee}\bigr)
\]
given by the polar pairing, 
we can identify $\mathbf A(W)$ with an element of $\bP(\Sym^{n(d-1)} V^{\vee}\bigr)$.
This gives an element in $\Sym^{n(d-1)} V^{\vee}$, defined up to a non-zero scalar, 
which is called \textbf{the associated form of $g_1,\dots,g_n$} by Alper and Isaev \cite[\S 2.2]{alper-isaev-binary}. 
Note that by construction, the associated form of $g_1,\dots,g_n$ is 
the element of $\Sym^{n(d-1)} V^{\vee}$
that is apolar to the codimension one subspace 
$(g_1,\dots,g_n)_{n(d-1)} \subset \Sym^{n(d-1)} V$. 
Classically, the associated form $\mathbf A(W)$
is known as the homogeneous Macaulay inverse system of 
$S_W$ with respect to the presentation $S_W=S/I_W$.

Since the Grassmannian $\Grass\bigl(\iota(m), \Sym^m V\bigr)$ admits a natural $\SL(V)$ action
and the morphism $H_m$ is equivariant on the locus where it is defined, we can ask the following:
\begin{question}\label{Q1}
For which $W$ and which $m$, is the $m^{th}$ Hilbert point of $S_W$ semistable with respect
to the $\SL(V)$-action? 
\end{question}
Our first result (Theorem \ref{T:main}) is a complete answer to Question \ref{Q1} for $m=d$ and for $W$ lying in 
the image of the gradient morphism. In the next two subsections, we describe this morphism and its image in more detail. 
The proof of Theorem \ref{T:main} will be given in Section \ref{S:proof}.

Since $\Res$ is an $\SL(V)$-invariant divisor, every point of $\Grass\bigl(n, \Sym^{d} V\bigr)_{\Res}$ is automatically 
$\SL(V)$-semistable. Hence we can ask  
\begin{question}\label{Q2} Suppose $d \leq m \leq n(d-1)$. 
Is $H_m$ a semistability preserving
morphism on the locus where it is defined? In particular, 
is the $m^{th}$ Hilbert point of $S_W$ semistable for every $[W] \in \Grass\bigl(n, \Sym^{d} V\bigr)_{\Res}$?
\end{question}
When $m=n(d-1)$, the above question is part of \cite[Question 3.1]{alper-isaev-binary}, which further asks
whether the induced morphism on the GIT quotients is an immersion.
Our second result (Theorem \ref{T:main-2}) is a complete answer to Question \ref{Q2}
for $m=n(d-1)$. We prove Theorem \ref{T:main-2} in Section \ref{S:proof-2}.

\subsection{Milnor algebra and its Hilbert points}\label{S:milnor}
As before, $S=\Sym V$.
The module of $\CC$-derivations of $S$ is naturally isomorphic to $V^\vee\otimes S$.
\begin{definition} Given $F\in \Sym^{d+1} V$, we define \textbf{the gradient point of $F$}
to be the subspace $\nabla F \subset \Sym^{d} V$ spanned by the first partial derivatives of $F$. 
That is, $\nabla F$ is the image of the natural
linear map $V^{\vee} \to \Sym^{d}V$ given by restricting to $V^{\vee} \times [F]$ the bilinear
differentiation map \[
V^{\vee} \times \Sym^{d+1}V \to \Sym^{d} V.
\] 
\end{definition}

Note that $\dim \nabla F=n$ if and only if $F\notin \Sym^{d+1} W$
for any proper subspace $W\subset V$. If $\dim \nabla F=n$, we will denote by $\nabla(F)$ the corresponding 
point of $\Grass\bigl(n, \Sym^{d} V\bigr)$.

The \emph{Jacobian ideal} of $F\in \Sym^{d+1} V$ is the ideal generated by the elements of $\nabla F$:
\begin{equation}\label{jacobian}
J_F:=I_{\nabla F}=(\partial F \mid \partial\in V^\vee),
\end{equation}
and \emph{the Milnor algebra of $F$} is
\[
M_F:=S/J_F.
\]
Concretely, if we choose a basis $\{x_1,\dots,x_n\}$ of $V$ and take the dual basis $\partial/\partial x_1, \dots, \partial/\partial x_n$ 
of $V^\vee$, then 
\[
\nabla F=\lspan\left\langle \partial F/\partial x_1, \dots, \partial F/\partial x_n\right\rangle,
\]
and the Milnor algebra of $F$ can be written explicitly as 
\begin{equation*}
M_F=\CC[x_1,\dots,x_n]/(\partial F/\partial x_1, \dots, \partial F/\partial x_n).
\end{equation*}
As we have already discussed, $\partial F/\partial x_1, \dots, \partial F/\partial x_n$
form a regular sequence in $S$ if and only if the locus given by
\[
\partial F/\partial x_1= \cdots= \partial F/\partial x_n=0
\]
is empty in $\bP V^{\vee}$ if and only if $F$ is smooth (the Jacobian Criterion).
In particular, if $F$ is smooth, then $M_F$ is a graded local Artinian Gorenstein $\CC$-algebra 
with socle in degree $\nu:=n(d-1)$.
As discussed in \S\ref{S:complete},
the interesting Hilbert points of $M_F$ occur only for $d \leq m \leq \nu$.

As also discussed in \S\ref{S:complete}, if $F$ is smooth, the $\nu^{th}$ Hilbert point of $M_F$ is a $1$-dimensional quotient of $\Sym^{\nu} V$, 
and so is an element of 
\[
\bP \left(\bigl(\Sym^{\nu} V\bigr)^\vee\right) \simeq \bP \bigl(\Sym^{\nu} V^{\vee}\bigr) \simeq \CC[\partial/\partial x_1, \dots, \partial/\partial x_n].
\] 
The corresponding element
of $\bP \bigl(\Sym^{\nu} V^{\vee}\bigr)$ is a line generated by
\textbf{the associated form of $F$}, as defined by Alper and Isaev in \cite[\S 2.2]{alper-isaev-assoc}.
(The $\nu^{th}$ Hilbert point of $M_F$ determines the associated form of $F$ up to a scalar, but one can recover
the form exactly using the condition that it takes value $1$ on the Hessian polynomial of $F$.)

The first main result of this paper is a characterization of the $\SL(V)$-semistability of the first non-trivial
Hilbert point of $M_F$ in terms of the $\SL(V)$-semistability of $F$. It is described in the next subsection. 

\subsection{The gradient morphism}
The association to a homogeneous form of its gradient point defines an $\SL(V)$-equivariant rational
map 
\[
\nabla \co \bP\bigl(\Sym^{d+1} V\bigr) \dashrightarrow \Grass\bigl(n, \Sym^{d} V\bigr),
\]
called \textbf{the gradient map}, cf. \cite[Section 2]{alper-isaev-binary}. 

Let $\Delta\subset \bP\bigl(\Sym^{d+1} V\bigr)$ be the $\SL(V)$-invariant divisor parameterizing
singular hypersurfaces and $\bP\bigl(\Sym^{d+1} V\bigr)_{\Delta}$ be the affine complement of $\Delta$.
Then $\nabla$ restricts to an $\SL(V)$-morphism between the affine varieties 
$\bP\bigl(\Sym^{d+1} V\bigr)_{\Delta}$ and $\Grass\bigl(n, \Sym^{d} V\bigr)_{\Res}$,
where the latter was defined in \S\ref{S:complete}.

Summarizing the above discussion, we have the following commutative diagram 
\begin{equation*}
\xymatrix{ 
\Sym^{d+1} V \ar[r] \ar@{-->}[d]    & \Hom(V^{\vee}, \Sym^{d} V)  \ar@{-->}[d] \\ 
\bP\bigl(\Sym^{d+1} V\bigr)  \ar@{-->}[r]^{\nabla \qquad} & \Grass(n, \Sym^{d} V) \\
\bP\left(\Sym^{d+1} V\right)_\Delta \ar[r] \ar@{^{(}->}[u] & \Grass(n, \Sym^{d} V)_{\Res} \ar@{^{(}->}[u]  
}
\end{equation*}

By definition, $\bP\bigl(\Sym^{d+1} V\bigr)_{\Delta}$ and  $\Grass\bigl(n, \Sym^{d} V\bigr)_{\Res}$ lie in the semistable
(with respect to the $\SL(V)$ action)
locus of $\bP\bigl(\Sym^{d+1} V\bigr)$ and $\Grass\bigl(n, \Sym^{d} V\bigr)$, respectively. In fact, 
all points in $\bP\bigl(\Sym^{d+1} V\bigr)_{\Delta}$ are automatically stable with respect to the $\SL(V)$ action
as long as $d\geq 2$ by Mumford's observation \cite[Proposition 4.2]{GIT}. Hence $\nabla$ preserves semistability on the locus of smooth forms.
We will prove that $\nabla$ always preserves semistability. More precisely, we have:
\begin{theorem}[Theorem \ref{T:main}]\label{T:main-restated}
Let $F\in \Sym^{d+1} V$.
Then $F$ is $\SL(V)$-semistable if and only if $\nabla(F)$ 
is a well-defined and $\SL(V)$-semistable point of $\Grass\bigl(n, \Sym^{d} V\bigr)$.
Suppose $F$ is stable. Then $\nabla(F)$ is polystable; moreover, $\nabla(F)$ is stable
if and only if $F\notin \Sym^{d+1} U+\Sym^{d+1} W$ for a non-trivial decomposition $V=U\oplus W$.
\end{theorem}

In the case of binary forms (i.e., $n=2$), the above result was established by Alper and Isaev 
\cite[Proposition 5.2]{alper-isaev-binary}; in the case of ternary forms (i.e., $n=3$) of degree $d+1\leq 9$,
the result was established by David Benjamin Lim (unpublished). 

\begin{remark}\label{R:stable}
It is easy to see that $\nabla$ does not preserve \emph{stability}. Indeed, the gradient point 
of the stable Fermat hypersurface \[x_1^{d+1}+\cdots+x_n^{d+1}\] is $\lspan\langle x_1^{d},\dots,x_n^{d}\rangle$, which is only strictly semistable. 

More generally, if $F=G(x_1,\dots,x_r)+H(x_{r+1},\dots,x_n)$ for some $1\leq r \leq n-1$, then $\nabla(F)$ is fixed
by the one-parameter subgroup of $\SL(n)$ acting with weights 
\[
(-(n-r),\dots,-(n-r), r, \dots, r)\]
 on $\{x_1,\dots,x_n\}$.
In particular, $\nabla(F)$ is strictly semistable even when $F$ is stable. 
\end{remark}

\section{GIT quotient of the gradient morphism}\label{S:applications}
In this section, we describe the applications of Theorem \ref{T:main-restated}, deferring its proof to Section \ref{S:proof}.
The main theorem implies that we have a cartesian diagram of $\SL(V)$-morphisms 
\begin{equation*}
\xymatrix{ 
\bP\bigl(\Sym^{d+1} V\bigr)^{ss}  \ar[r]^{\nabla \qquad} & \Grass\bigl(n, \Sym^{d} V\bigr)^{ss} \\
\bP\bigl(\Sym^{d+1} V\bigr)_\Delta \ar[r] \ar@{^{(}->}[u] & \Grass\bigl(n, \Sym^{d} V\bigr)_{\Res}\ , \ar@{^{(}->}[u],  
}
\end{equation*}
where the vertical arrows are saturated open inclusions of affines. After forming the GIT quotients, we obtain 
a cartesian diagram
\begin{equation}\label{diagram}
\begin{aligned}
\xymatrix{ 
\bP\bigl(\Sym^{d+1} V\bigr)^{ss}\gitq \SL(V)  \ar[r]^{\overline{\nabla}\qquad} & \Grass\bigl(n, \Sym^{d} V\bigr)^{ss} \gitq \SL(V) \\
\bP\bigl(\Sym^{d+1} V\bigr)_\Delta \gitq \SL(V) \ar[r]^{\widetilde{\nabla}\qquad } \ar@{^{(}->}[u] 
& \Grass\bigl(n, \Sym^{d} V\bigr)_{\Res} \gitq \SL(V), \ar@{^{(}->}[u]
}
\end{aligned}
\end{equation}
where the top arrow $\overline{\nabla}:=\nabla_{/\SL(V)}$ is a projective morphism on the GIT quotients induced by $\nabla$,
and the bottom arrow $\widetilde{\nabla}$ is a morphism of affine GIT quotients. In what follows,
we show that $\overline{\nabla}$ is finite and birational onto its image, while 
$\widetilde{\nabla}$ is finite and injective.

Recall that by a result of Donagi \cite[Proposition 1.1]{donagi}, two hypersurfaces are projectively equivalent if 
and only if their gradient points are projectively equivalent. This however does not immediately imply that $\overline{\nabla}$ 
is injective because distinct $\SL(V)$-orbits can be identified when passing to a GIT quotient. 
In Proposition \ref{P:almost-embedding} below, we 
prove that our main theorem does imply injectivity of $\overline{\nabla}$
on the \emph{stable} locus. 

Alper and Isaev asked whether $\overline{\nabla}$ is in fact a closed embedding \cite[Question 3.3]{alper-isaev-binary}.
They note that establishing that $\widetilde{\nabla}$ is a closed embedding is one of the two steps sufficient to prove their main 
conjecture (namely, \cite[Conjecture 1.1]{alper-isaev-binary}). Proposition \ref{P:almost-embedding} 
implies that $\widetilde{\nabla}$ is a composition of a closed embedding 
and a bijective normalization morphism. 

We introduce the following notation. Given $F\in \bP \bigl(\Sym^{d+1} V\bigr)$, we denote by $T_F$ the tangent space at $F$ 
and by $N_{F}$ the normal space to the $\SL(V)$-orbit at $F$.  
Let $T_{\nabla F}$ be the tangent space at $\nabla F \in \Grass\bigl(n, \Sym^{d} V\bigr)$
and $N_{\nabla F}$ be the normal space to the $\SL(V)$-orbit at $\nabla F$.

\begin{prop}\label{P:almost-embedding}  \hfill
\begin{enumerate}
\item The morphism $\overline{\nabla}$ is a finite morphism of projective normal varieties. 
\item The restriction of \ $\overline{\nabla}$ to the stable locus is injective. 
\item The morphism $\widetilde{\nabla}$ is finite and injective. In particular, 
$\widetilde{\nabla}$ is a normalization of its image in $\Grass\bigl(n, \Sym^{d} V\bigr)_{\Res}$.
\item Given a stable point $F\in \bP\bigl(\Sym^{d+1} V\bigr)^{s}$, the map $N_{F} \to N_{\nabla F}$ 
induced by $\nabla$ is injective. 
 \end{enumerate}
\end{prop}

\begin{proof}
(1) The fact that the morphism exists follows from Theorem \ref{T:main-restated}.
By Kempf's descent lemma, both GIT quotients are projective varieties of Picard number $1$.
Hence $\overline{\nabla}$ is a finite morphism. The normality of the domain and target follows 
from the preservation of normality under GIT quotients. 

(2) We now establish injectivity of $\overline{\nabla}$ on the stable locus.
Suppose $\overline{\nabla}(F_1)=\overline{\nabla}(F_2)$ for two stable hypersurfaces $F_1$
and $F_2$. 
By Theorem \ref{T:main-restated}, the $\SL(V)$-orbits of $\nabla F_1$ and $\nabla F_2$ 
are closed in the semistable locus of the Grassmannian. Since they are identified in the GIT quotient, the two orbits must be equal. 
Acting by an element of $\SL(V)$, we thus
can assume that $\nabla F_1=\nabla F_2$. It is easy to see then that 
\[
\nabla(sF_1+tF_2)=\nabla F_1 \quad \text{for all $[s:t]\in \bP^1 \setminus D$, where $D$ is some finite set.}
\]
(Up to this point, our argument followed Donagi's proof of \cite[Proposition 1.1]{donagi}. In what follows, 
we replace Donagi's deformation theory argument by the already established Part (1) of this proposition.)

Since $\nabla(sF_1+tF_2)$ is semistable for $[s:t]\in \bP^1 \setminus D$,
we conclude that $\bP^1 \setminus D$ lies in the semistable locus and is contracted by $\nabla$ to a point. 
Moreover, the generic point of $\bP^1$ is stable because $F_1$ is stable. 
Since the fibers of $\overline{\nabla}$ are finite, we conclude that $\bP^1 \setminus D$ must
be contracted to a point in the GIT quotient $\bP \bigl(\Sym^{d+1} V\bigr)^{ss} \gitq \SL(V)$. 
This shows that $\bP^1 \setminus D$ lies entirely in an $\SL(V)$-orbit, proving that $F_1$ and $F_2$ are in the same $\SL(V)$-orbit.

(3) We note that $\widetilde{\nabla}$ is finite by base change. Since smooth hypersurfaces of degree $\geq 3$ are stable,
we conclude that $\widetilde{\nabla}$ is injective by Part (2) if $d\geq 2$; 
if $d=1$, the domain of $\widetilde{\nabla}$ is a point. By normality of $\bP\bigl(\Sym^{d+1} V\bigr)_\Delta \gitq \SL(V)$,
it follows that $\widetilde{\nabla}$ is a normalization of its closed image. 

(4) Suppose that for a stable hypersurface $F$,
some non-zero vector $v\in N_F$ maps to $0\in N_{\nabla F}$.
Since the differential of $\nabla$ restricts to a surjective map on the tangent spaces between the $\SL(V)$-orbits, 
we can find a lift of $v$ to $T_{F}$ that maps to $0$ in $T_{\nabla F}$.
This lift corresponds to a first-order deformation $F+\epsilon G$, where $G$ is some
stable form and $\epsilon^2=0$. The induced first order deformation of $\nabla F$ is an element of \[
\Hom\left(\nabla F, \ \Sym^{d} V \big{/} \nabla F\right)\]
given by 
\[
\partial F/\partial x_i \mapsto \partial G/\partial x_i \mod \nabla F, \quad \text{ for $i=1,\dots, n$}.
\] 
For this first order deformation of $\nabla F$ to be $0$ in the tangent space of the Grassmannian, we must have 
$\nabla G \subset \nabla F$, which implies $\nabla G=\nabla F$ since $\dim \nabla G=\dim \nabla F=n$.
As we have already seen, $\nabla G=\nabla F$ for two stable forms $F$ and $G$
implies that the semistable locus of the line joining $F$ and $G$ lies in the same $\SL(V)$-orbit,
which of course means that $v=0$. A contradiction! 
\end{proof}

\begin{remark}
Suppose $F$ is a stable hypersurface. Then $\nabla F$ is polystable by Theorem \ref{T:main-restated}. By the Luna's \'{e}tale slice theorem,
in a neighborhood of $F$, we have that $\overline{\nabla}$ \'{e}tale locally looks like
\[
N_{F}\gitq \Stab(F) \to N_{\nabla F}\gitq \Stab(\nabla F).
\]
If $\nabla$ is stabilizer preserving at $F$, that is $\Stab(F)=\Stab(\nabla F)$, then Proposition \ref{P:almost-embedding}
implies that $\overline{\nabla}$ is unramified at $F$. However, $\nabla$ is not stabilizer preserving even on 
the stable locus as Remark \ref{R:stable} shows. 
In general, it seems to be a difficult problem to control stabilizers of hypersurfaces
and, especially, of their gradient points.
\end{remark}

\section{Semistability of the gradient point}  
\label{S:proof}

\subsection{Semistability of linear spaces of homogeneous forms}
\label{S:grassmannian}
We begin by reviewing the Hilbert-Mumford Numerical Criterion 
for the Grassmannian $\Grass\bigl(k, \Sym^{m} V\bigr)$.

In what follows, we always let $\lambda$ be 
a one-parameter subgroup ($1$-PS) of $\SL(V)$ acting diagonally on 
a basis $\{x_1,\dots, x_n\}$
with weights 
\[
\lambda_1 \leq \cdots \leq \lambda_n \qquad \text{(N.B. $\sum_{i=1}^n \lambda_i=0$).}
\] 
\subsubsection{$\lambda$-ordering on monomials} 
The $\lambda$-weight of a monomial $x_1^{a_1}\cdots x_n^{a_n} \in \Sym V$ is defined to be
\[
w_{\lambda}(x_1^{a_1}\cdots x_n^{a_n}):=\sum_{i=1}^n a_i\lambda_i.
\]
The $\lambda$-weight induces a monomial ordering $<_{\lambda}$ on the monomials in $\Sym^m V$
given by the $\lambda$-weight, with ties broken lexicographically. Precisely, 
for two degree $m$ monomials with multi-degrees $(a_1,\dots, a_n)$ and $(b_1,\dots, b_n)$, we set
\[
x_1^{a_1}\cdots x_n^{a_n} <_{\lambda} x_1^{b_1}\cdots x_n^{b_n} 
\]
if and only if 
\begin{itemize} 
\item either $w_{\lambda}(x_1^{a_1}\cdots x_n^{a_n}) < w_{\lambda}(x_1^{b_1}\cdots x_n^{b_n})$,
\item or $w_{\lambda}(x_1^{a_1}\cdots x_n^{a_n}) = w_{\lambda}(x_1^{b_1}\cdots x_n^{b_n})$, and 
for some $r=1,\dots, n-1$, we have $a_i=b_i$ for $i=1,\dots, r$, and $a_{r+1}>b_{r+1}$.
\end{itemize}
Set $N=\binom{n+m-1}{m}$ and let 
\[X_1 <_{\lambda} \cdots <_{\lambda} X_N\] 
be the degree $m$ monomials in the variables $\{x_1,\dots, x_n\}$, ordered by $<_{\lambda}$.

Given $F\in \Sym^m V$, \emph{the initial monomial of $F$ with respect to $\lambda$}, denoted $\init_{\lambda}(F)$, 
is the smallest, with respect to  $<_{\lambda}$, 
monomial appearing with a non-zero coefficient in the expansion of $F$ in terms of the monomials $X_1,\dots, X_N$.

\subsubsection{$\lambda$-ordering on wedges}\label{order-grassmann}
The monomial ordering $<_{\lambda}$ induces an ordering on the decomposable elements 
of the form $X_{i_1} \wedge \cdots \wedge X_{i_k}$ in  $\wedge^k \Sym^m V$. 
First, we define the $\lambda$-weight of $X_{i_1} \wedge \cdots \wedge X_{i_k}$ to be 
\[
w_{\lambda}(X_{i_1} \wedge \cdots \wedge X_{i_k}):=\sum_{r=1}^k w_{\lambda}(X_{i_r}).
\]
Next, for two multi-indices $i_1<\cdots<i_k$ and $j_1<\cdots<j_k$, we set 
\[
X_{i_1} \wedge \cdots \wedge X_{i_k} <_{\lambda} X_{j_1} \wedge \cdots \wedge X_{j_k}
\]
if and only if 
\begin{itemize} 
\item either $w_{\lambda}(X_{i_1} \wedge \cdots \wedge X_{i_k})< w_{\lambda}(X_{j_1} \wedge \cdots \wedge X_{j_k})$,
\item or $w_{\lambda}(X_{i_1} \wedge \cdots \wedge X_{i_k})= w_{\lambda}(X_{j_1} \wedge \cdots \wedge X_{j_k})$, and 
for some $r=1,\dots, k-1$, we have $i_s=j_s$ for $s=1,\dots, r$, and $i_{r+1}<j_{r+1}$.
\end{itemize}

\begin{definition}
Given $[W]\in  \Grass\bigl(k, \Sym^{m} V\bigr)$, let $X_{i_1}, \dots, X_{i_k}$ be the $k$ distinct initial monomials of the elements in 
$W$ with respect to a $1$-PS $\lambda$.
We call $X_{i_1}\wedge \cdots \wedge X_{i_k}$ \emph{the initial Pl\"ucker coordinate of $W$ with respect to $\lambda$}.

\end{definition}

\begin{lemma}[Hilbert-Mumford Numerical Criterion for Grassmannians]\label{HM-grass}
A point $[W]\in  \Grass\bigl(k, \Sym^{m} V\bigr)$ is unstable (resp., strictly semistable)
with respect to $\lambda$
if and only if the $\lambda$-weight of the initial Pl\"ucker coordinate of $W$ with respect to $\lambda$ is positive
(resp., zero).
\end{lemma}
\begin{proof}
Clearly, $X_{i_1}\wedge \cdots \wedge X_{i_k}$ has the least $\lambda$-weight among all Pl\"ucker coordinates (with 
respect to the basis $\{X_1,\dots, X_N\}$ of $\Sym^{m} V$) 
that are non-zero on $[W]$. The claim follows from the usual Hilbert-Mumford Numerical Criterion applied to $\wedge^k \Sym^m V$.
\end{proof}

We will need the following observation:
\begin{lemma}\label{L:sub}
Suppose $W=\lspan \left\langle g_1, \dots, g_k \right\rangle \in \Grass\bigl(k, \Sym^{m} V\bigr)$.
Suppose $\lambda$ is a $1$-PS of $\SL(V)$  acting diagonally on $\{x_1,\dots, x_n\}$
with weights $\lambda_1 \leq \cdots \leq \lambda_n$.
Consider a change of coordinates
\begin{equation}\label{E:sub}
\begin{aligned}
x_1& \mapsto x_1+c_{12}x_2+\cdots+ c_{1n}x_n \\
x_2&\mapsto\phantom{{}=1111} x_2+\cdots+ c_{2n}x_n \\
\vdots \\
x_n&\mapsto \phantom{{}=1111111111111111}x_n
\end{aligned}
\end{equation}
Let $g'_i(x_1,\dots,x_n)=g_i(x_1+c_{12}x_2+\cdots+ c_{1n}x_n, \ x_2+\cdots+ c_{2n}x_n, \ \dots, \ x_n)$
and
\[W':=\operatorname{span}\left\langle g'_1, \dots, g'_k \right\rangle.\]
Then the initial Pl\"ucker coordinates of $W$ and $W'$ with respect to $\lambda$ are the same.
In particular, if $W$ is $\lambda$-unstable (resp., $\lambda$-strictly semistable), then $W'$ is also $\lambda$-unstable
(resp., $\lambda$-strictly semistable). 
\end{lemma}
\begin{remark} The above lemma is closely related to a more general result of Kempf,
who proved that if $\lambda$ is a worst destabilizing $1$-PS, then all conjugates
of $\lambda$ by the elements of the unipotent radical 
of the associated parabolic subgroup $P(\lambda)$ are also worst destabilizing $1$-PS's
\cite[Theorem 2.2]{kempf}.
\end{remark}

\begin{proof}[Proof of Lemma \ref{L:sub}]
Let $A$ be the matrix of $\{g_1,\dots, g_k\}$ in the basis $\{X_1,\dots, X_N\}$, such that the $i^{th}$ row
of $A$ is the coordinate vector of $g_i$.
Then the Pl\"ucker coordinates of $W$ with respect to $\{X_1,\dots, X_N\}$ correspond
to the $k\times k$ minors of $A$, which are in turn ordered by $<_{\lambda}$ as defined in \S\ref{order-grassmann}.

Notice that the upper triangular transformation \eqref{E:sub} induces an upper triangular transformation 
on the degree $m$ monomials:
\begin{equation}\label{induced}
X_i \mapsto X_i+C_{i i+1} X_{i+1}+\cdots+C_{i N} X_N.
\end{equation}
It follows that the matrix $A'$ of $\{g'_1,\dots, g'_k\}$ 
is obtained from $A$ by the following column operations: 
\begin{itemize}
\item a multiple of the $i^{th}$ column is added to the $j^{th}$ column
only if $i<j$. 
\end{itemize}
Evidently, the initial Pl\"ucker coordinate 
remains unchanged under these column operations
and the claim follows.
\end{proof}

\subsection{{Proof of Theorem \ref{T:main-restated}}}

It is straightforward to see that unstable polynomials have unstable gradient points.
Indeed, suppose $F$ is destabilized by a $1$-PS acting diagonally 
on a basis $x_1,\dots, x_{n}$ with weights \[
\lambda_1 \leq \cdots \leq \lambda_n.
\]
Then all monomials of $F(x_1,\dots,x_{n})$ have positive $\lambda$-weight. It follows
that all monomials of $\partial F/\partial x_i$ have weight greater than $-\lambda_i$. Hence
all non-zero Pl\"ucker coordinates of $\nabla F$ have weight greater than 
\[
\sum_{i=1}^n (-\lambda_i)=0.
\]
This shows that $\nabla F$ is also destabilized by $\lambda$. 

We now proceed to prove the reverse implication, which is the heart of the theorem. 
To begin, if $\nabla F$ is not $n$-dimensional, then in some coordinate system we have $\partial F/\partial x_1=0$. 
It follows that $F$ is a polynomial in $x_2,\dots, x_n$ and so is destabilized by the $1$-PS with weights
$(-(n-1), 1, \dots, 1)$. 

Suppose $\nabla F$ is an unstable point in $\Grass\bigl(n, \Sym^{d}V\bigr)$.
Let $\lambda$ be a destabilizing $1$-PS acting diagonally 
on a basis $x_1,\dots, x_{n}$ of $V$ with weights 
\[
\lambda_1 \leq \cdots \leq \lambda_n.
\]

The following is the first of the two key results used in our proof of Theorem \ref{T:main-restated}:
\begin{lemma}\label{key-1}
After a change of variables as in \eqref{E:sub}, we can assume that the initial 
monomials $\init_{\lambda}(\partial F/\partial x_i)$ are distinct for $i\in \{1,\dots,n\}$.
\end{lemma}
\begin{proof}[Proof of Lemma \ref{key-1}]
We will apply Lemma \ref{L:sub} with $k=n$ and $m=d$.
In particular, $\{X_1,\dots, X_N\}$, where $N=\binom{n+d-1}{d}$, will be the set 
of monomials in $\Sym^{d} V$ ordered by $<_{\lambda}$. In what follows,
we will often write $F_i$ to denote $\partial F/\partial x_i$.

It follows from the proof of Lemma \ref{L:sub} that an upper-triangular substitution \eqref{E:sub}
transforms the matrix of $\{F_1, \dots, F_n\}$
in the basis $\{X_1,\dots, X_N\}$ by some sequence of the following operations:
\begin{itemize}
\item For $a<b$, add a multiple of the $a^{th}$ column to the $b^{th}$ column.
\item For $c<d$, add a multiple of the $c^{th}$ row to the $d^{th}$ row.
\end{itemize}
The point of the present lemma is that by choosing a sequence of the above operations carefully, 
we can ignore column operations and choose row operations so that the initial monomials
of the $n$ rows become distinct. For the lack of imagination needed \emph{to explain} how to do so, 
we proceed \emph{to prove} the claim formally.

Suppose $X_{i_1}\wedge \cdots \wedge X_{i_n}$ is the initial Pl\"ucker coordinate of $\nabla F$
with respect to $\lambda$, where $1\leq i_1<\cdots <i_n \leq N$. Note that by Lemma \ref{L:sub}, 
the initial Pl\"ucker coordinate of $\nabla F$ remains constant
under the change of coordinates \eqref{E:sub}.

For $r\leq n$, 
we are going to prove
that there exist indices $\{j_1,\dots, j_r\}\subset \{1,\dots, n\}$ and 
an upper-triangular change of coordinates \eqref{E:sub} such that:
\begin{enumerate}
\item
$\init_{\lambda} (F_{j_s})=X_{i_s}$ for all $s=1,\dots, r$.

\item 
$
\init_{\lambda} (F_{j})\notin \{X_{i_1}, \dots, X_{i_r}\}
$
for all $j \in \{1,\dots, n\} \setminus \{j_1,\dots, j_r\}$;
\end{enumerate}

The base case of $r=0$ is vacuous. Suppose the claim has been established for some $r$. 
Then the smallest (with respect to $<_{\lambda}$) initial monomial of 
\[
\lspan \langle F_j \mid j \in \{1,\dots, n\} \setminus \{j_1,\dots, j_r\}\rangle
\]
is $X_{i_{r+1}}$. In particular, there exists \emph{the smallest index} $j_{r+1} \in \{1,\dots, n\} \setminus \{j_1,\dots, j_r\}$
such that 
\[
\init_{\lambda} (F_{j_{r+1}})=X_{i_{r+1}}.
\] 
Without loss of generality, we can assume that $X_{i_{r+1}}$ occurs in $F_{j_{r+1}}$ with coefficient $1$.
For every $j > j_{r+1}$, let $c_{j_{r+1} j}$ be the coefficient of $X_{i_{r+1}}$ in $F_{j}$.

Consider the change of variables 
\begin{equation*}
\begin{aligned}
x_j& \mapsto x_j \text{ for all $j\neq j_{r+1}$}, \\
x_{j_{r+1}} & \mapsto x_{j_{r+1}}-\sum_{j>j_{r+1}} c_{j_{r+1} j} \ x_j \ . 
\end{aligned}
\end{equation*}

Set 
\begin{align*}
G&:=F(x_1,\dots, \ x_{j_{r+1}}-\sum_{j>j_{r+1}} c_{j_{r+1} j} \ x_j\ , \ \ \dots, \ x_n), \quad \text{and} \\
\widetilde{F}_i&:=F_i(x_1,\dots, \ x_{j_{r+1}}-\sum_{j>j_{r+1}} c_{j_{r+1} j} \ x_j\ , \ \ \dots, \ x_n),
\text{ for $i=1,\dots,n$}.
\end{align*}
Note using \eqref{induced} that $\init_{\lambda}(\widetilde{F}_i)=\init_{\lambda}(F_i)$ for all $i=1,\dots,n$.
We compute
\[
\frac{\partial G}{\partial x_j}=\widetilde F_j, \text{ for all $j\leq j_{r+1}$}
\]
and 
\[
\frac{\partial G}{\partial x_j}=\widetilde F_j -c_{j_{r+1} j} \ \widetilde F_{j_{r+1}},
\text{ for all $j> j_{r+1}$.}
\]

Note that the initial monomial of 
$\partial G/\partial x_{j_s}$ is still $X_{i_s}$, for all $s=1,\dots, r$, because
\[
X_{i_s} <_{\lambda} X_{i_{r+1}}=\init_{\lambda}(\widetilde F_{j_{r+1}}).
\] 
Clearly, we still have \[
\init_{\lambda} \left(\partial G/\partial x_{j_{r+1}}\right)=X_{i_{r+1}}.
\] 

The coefficient of $X_{i_{r+1}}$ in $\partial G/\partial x_j$ remains $0$ for all $j< j_{r+1}$ such that $j\notin \{j_1,\dots, j_{r}\}$,
and, by the choice of the scalars $c_{j_{r+1} j}$, the coefficient of $X_{i_{r+1}}$ in $\partial G/\partial x_j$ becomes $0$ for all 
$j>j_{r+1}$ such that $j\notin \{j_1,\dots, j_{r}\}$. 
This means that $\init_{\lambda} (\partial G/\partial x_j)\notin \{X_{i_1}, \dots, X_{i_{r+1}}\}
$
for all $j \in \{1,\dots, n\} \setminus \{j_1,\dots, j_{r+1}\}$.
The induction step follows. 
\end{proof}

Applying Lemma \ref{key-1}, we continue with the proof of Theorem \ref{T:main-restated}
under the assumption that $\init_{\lambda}(F_i)$ are all distinct. In this case, 
the initial Pl\"ucker coordinate of $\nabla F$ is precisely
\[
\init_{\lambda}(F_1)\wedge \cdots \wedge \init_{\lambda}(F_n).
\]

Since, by assumption, $\lambda$ destabilizes $\nabla F$
in $\Grass\bigl(n, \Sym^{d} V\bigr)$, we have by Lemma \ref{HM-grass} that
\[
\sum_{i=1}^{n} w_\lambda\bigl(\init_{\lambda} (F_i)\bigr)>0.
\]
We
can choose rational numbers $\mu'_1,\dots,\mu'_n$ such that $\sum_{i=1}^n \mu'_i=0$
and, for all $i$, we have
\[
w_\lambda\bigl(\init_{\lambda} (F_i)\bigr) >\mu'_i.
\]
Equivalently, the $\lambda$-weight of every monomial in $\partial F/\partial x_i$ is greater than $\mu'_i$.

It follows that for every monomial $x_1^{d_1}\cdots x_n^{d_n}$ appearing with a non-zero coefficient
in $F$, and for every $i$,
we either have $d_i=0$ or 
\[
\lambda_1 d_1+\cdots+\lambda_i (d_i-1)+\cdots +\lambda_n d_n>\mu'_i.\]

Hence, for every $i$, either $d_i=0$ or 
\[
\lambda_1 d_1+\cdots+\lambda_id_i +\cdots + \lambda_n d_n>\mu'_i+\lambda_i.
\]
Set $\mu_i:=\mu'_i+\lambda_i$. Notice that $\sum_{i=1}^n \mu_i=0$.

The following key lemma now implies that the $T$-state of $F$, with respect to the 
torus $T$ in $\SL(n)$ acting diagonally on $\{x_1,\dots,x_n\}$, lies to one side of a hyperplane 
passing through the barycenter and hence is $T$-unstable. This finishes the proof of the 
first part of Theorem \ref{T:main} (namely, the fact that $\nabla F$ is semistable if and only if $F$ is).

\begin{lemma}\label{key-2}
Suppose $L(z_1,\dots,z_n)$ is a $\QQ$-linear function that vanishes at the barycenter of 
the $n$-simplex 
\[
\Delta_n:=\{(z_1,\dots,z_n)\mid \sum_{i=1}^n z_i=d+1, \ z_i\geq 0\}.
\] Suppose 
$\mu_1, \dots, \mu_n$ are rational numbers such that $\sum_{i=1}^n \mu_i=0$. 
Let 
\[
S_i=\{(z_1,\dots, z_n)\in \Delta_n \mid L(z_1,\dots,z_n)> \mu_i \quad \text{or} \quad z_i=0\}.
\]
Then there exists a $\QQ$-linear function that vanishes at the barycenter and 
that assumes positive values at all points of $S_1\cap \cdots \cap S_n$.
\end{lemma}
\begin{proof}[Proof of Lemma \ref{key-2}]
We have that 
\[
z_iL(z_1,\dots, z_n)\geq \mu_i z_i
\]
for every $(z_1,\dots,z_n)\in S_i$. 
Moreover, the inequality is strict if $z_i > 0$.
It follows that
\[
(z_1+\cdots+z_n)L(z_1,\dots,z_n)>\sum_{i=1}^n \mu_iz_i,
\]
or 
\[
(d+1)L(z_1,\dots,z_n)>\sum_{i=1}^n \mu_iz_i,
\]
for all $(z_1,\dots, z_n)\in S_1\cap \cdots \cap S_n$.

Clearly, \[
M(z_1,\dots,z_n):=(d+1)L(z_1,\dots,z_n)-\sum_{i=1}^n \mu_iz_i
\]
is the requisite linear functional. 
\end{proof}

Finally, suppose $F$ is stable but $\nabla F$ is strictly semistable with respect to some
$1$-PS $\lambda$. The argument above in the case when $\lambda$ is a destabilizing $1$-PS
of $\nabla F$
goes through after all strict inequalities are replaced by non-strict inequalities. 
In particular, after applying Lemma \ref{key-1}, we can use Lemma \ref{key-2} to conclude that  
the state of $F$ lies in the non-negative half-space with respect
to the linear function
\[
M(z_1,\dots,z_n)=(d+1)L(z_1,\dots,z_n)-\sum_{i=1}^n \mu_iz_i=\sum_{i=1}^n \bigl(d\lambda_i -\mu'_i\bigr) z_i.
\]
Since $F$ is stable, we must have $M\equiv 0$, or, equivalently, $\mu'_i=d\lambda_i$ for every $i=1,\dots, n$.
Hence,
\[
w_{\lambda}\bigl(\init_{\lambda} \left(\partial F/\partial x_i\right)\bigr) 
= d\lambda_i, \ \text{for all $i=1,\dots, n$.} 
\]
Suppose $r$ is the smallest index such that $\lambda_{r+1}=\cdots=\lambda_n$.
We claim that $F(x_1,\dots,x_n)=G(x_1,\dots, x_{r})+H(x_{r+1},\dots,x_n)$. Suppose not.
Then for some $s\leq r$ and $t\geq r+1$, there 
exists a monomial of degree $d+1$ that is divisible by $x_sx_t$ and  
that occurs with a non-zero coefficient in $F$. Then $\partial F/\partial x_t$ has a 
monomial divisible by $x_s$. In particular, 
\[
d\lambda_{t}=w_{\lambda}\bigl(\init_{\lambda} \left(\partial F/\partial x_t\right)\bigr) 
\leq \lambda_s+(d-1)\lambda_n<d\lambda_n.
\]
A contradiction!

Suppose $F$ is stable and $\nabla(F)$ is strictly semistable.
It remains to prove that $\nabla(F)$ has a closed $\SL(V)$-orbit.
By what has already been proven, we can choose a decomposition $V=W_1\oplus \cdots \oplus W_r$, where $r\geq 2$, 
such that $F=G_1+\cdots+G_r$, where $G_i\in \Sym^{d+1} W_i$ for $i=1,\dots, r$, and such that
$G_i$'s are not non-trivial sums of two polynomials in disjoint sets of variables. Note that $G_i$ is stable with respect to 
the $\SL(W_i)$ action for each $i=1,\dots, r$,
because $F$ is stable with respect to $\SL(V)$. 

Note that
\[
\nabla F=\nabla G_1\oplus \cdots \oplus \nabla G_r \subset \Sym^{d} W_1\oplus\cdots \oplus \Sym^{d}W_r \subset \Sym^{d} V
\]
is stabilized by every $1$-PS such that $W_i$'s are its eigenspaces. Choose a $1$-PS $\lambda$ such that
$W_i$'s are distinct eigenspaces of $\lambda$. Then the centralizer of $\lambda$ in $\SL(V)$ is 
\[
C_{\SL(V)}(\lambda)=\bigl(\GL(W_1)\times \cdots \times \GL(W_r)\bigr) \cap \SL(V).
\]
It follows by Luna's results \cite[Corollaire 2 and Remarque 1]{luna-adherences}, that the $\SL(V)$-orbit of $\nabla F$ is closed if and only 
if the $C_{\SL(V)}(\lambda)$-orbit of $\nabla F$ is closed. 

Set 
$n_i=\dim W_i$. Suppose $\mu$ is a $1$-PS of $C_{\SL(V)}(\lambda)$
acting on some basis of $W_i$ with weights $\{\mu_j^{i}\}_{j=1}^{n_i}$, for each $i=1,\dots, r$.
Let $\widetilde{\mu}$ be the $1$-PS of $C_{\SL(V)}(\lambda)$
acting on the same basis of $W_i$ with weights $\{\widetilde{\mu}_j^{i}\}_{j=1}^{n_i}$, where
\[
\widetilde{\mu}_j^{i}=\mu_j^{i}-\frac{1}{n_i} \sum_{j=1}^{n_i} \mu_j^{i}.
\]
The renormalized $\widetilde{\mu}$ is a $1$-PS of $\SL(W_1)\times \cdots \times \SL(W_r)$.
Notice that the actions of both $\mu$ and $\widetilde{\mu}$ on $\nabla F$ 
are identical because each Pl\"ucker coordinate of $\nabla F$ contains exactly $n_i$ vectors
from $\Sym^{d} W_i$. 
The orbit of $\nabla F$ under $\widetilde{\mu}$ is closed because we have already established
that the orbit of $\nabla G_i$ is closed under the $\SL(W_i)$ action. 
We conclude that the orbit of $\nabla F$ under $\mu$ is closed as well. This finishes the proof of Theorem \ref{T:main-restated}.

\section{Semistability of the associated form} 
\label{S:proof-2}
We keep notation of Subsections \ref{S:GIT} and \ref{S:complete}, but recall the necessary definitions for the reader's convenience:
As before, $V$ is a $\CC$-vector space of dimension $n$ and $S=\Sym V$ is the algebra of polynomials
on $V^{\vee}$.  If $W=\lspan \langle g_1,\dots, g_n\rangle$ is generated by a regular sequence of $n$ elements
in $\Sym^{d} V$, then the ideal $I_W=(g_1,\dots,g_n)$ defines a graded local Artinian Gorenstein $\CC$-algebra
\[
S_W=S/I_W \simeq \CC[x_1,\dots,x_n]/(g_1,\dots,g_n).
\]
The socle degree of $S_W$ is $\nu=n(d-1)$. Regarding the degree $\nu$ graded piece of $S_W$ as an element
of \[
\bP\left(\bigl(\Sym^{\nu} V\bigr)^{\vee}\right)\simeq \bP \Sym^{\nu} V^{\vee},
\] we obtain \textbf{the associated form} of $(g_1,\dots,g_n)$ as the corresponding element 
\[
\mathbf A (g_1,\dots,g_n) \in \bP\bigl(\Sym^{\nu} V^{\vee}\bigr)= \bP\bigl(\CC[\partial/\partial x_1, \dots, \partial/\partial x_n]_{\nu}\bigr),\]
where $\CC[\partial/\partial x_1, \dots, \partial/\partial x_n]_{\nu}$ is identified with $\Hom(\CC[x_1,\dots,x_n]_{\nu}, \CC)$
via the polar pairing. 

In this section, we answer in affirmative the semistability part of \cite[Question 3.1]{alper-isaev-binary} 
by proving:
\begin{theorem}[Theorem \ref{T:main-2}]
\label{T:associated}
Suppose that $\{g_1,\dots,g_n\}$ is a regular sequence of elements in $\Sym^d V$. Then 
the associated form $\mathbf A (g_1,\dots,g_n) \in \bP\bigl(\Sym^{n(d-1)} V^{\vee}\bigr)$ is semistable with respect to the $\SL(V)$-action.
In particular, for every smooth form $F\in \Sym^{d+1} V$, 
the associated form
\[
\mathbf A(F)=\mathbf A \bigl(\partial F/\partial x_1, \dots, \partial F/\partial x_n\bigr)\]
is semistable.
\end{theorem}

In the case of binary forms (i.e., $n=2$), the above result was established by Alper and Isaev, cf.
\cite[Proposition 4.1]{alper-isaev-binary}. Moreover, Alper and Isaev proved that the associated form $\mathbf A(F)$
is smooth (hence stable) for a \emph{generic} $F\in \Sym^{d+1} V$, with $d\geq 2$ \cite[Proposition 4.3]{alper-isaev-assoc}.

Theorem \ref{T:associated} follows from the following more technically stated result:
\begin{prop}\label{P:assoc} Suppose that $\{g_1,\dots,g_n\}$ is a regular sequence of elements in $\Sym^d V$ 
generating the ideal $I=(g_1,\dots,g_n)$ in $S$.
Suppose $\lambda$ is a $1$-PS of $\SL(V)$ acting with weights \[\lambda_1 \leq \cdots \leq \lambda_n\] on a basis $\{x_1,\dots, x_n\}$ of $V$. 
Consider the resulting monomial order $<_{\lambda}$ on the monomials in the variables $x_1,\dots,x_n$,
as described in Subsection \ref{S:grassmannian}. Then the set of the initial monomials of $I_{n(d-1)}$ contains all monomials of degree $n(d-1)$ with the
exception of a single monomial $m_0$ such that 
\[
m_0 \geq_{\lambda} x_1^{d-1}\cdots x_n^{d-1}.
\]
\end{prop}

Before we prove Proposition \ref{P:assoc}, we explain how Theorem \ref{T:associated} follows from it.
\begin{proof}[Proof of Theorem \ref{T:associated}]
Let $I=(g_1,\dots,g_n)$ be the ideal in $S=\Sym V$ generated by a regular sequence in $\Sym^d V$. 
Our goal is to prove that the $n(d-1)^{st}$ Hilbert point of $S/I$ is semistable. Equivalently, we need
to prove that $I_{n(d-1)} \subset \Sym^{n(d-1)}V$ is semistable as a point in $\Grass \bigl(\dim I_{n(d-1)}, \Sym^{n(d-1)}V\bigr)$.

Suppose $\lambda$ is a $1$-PS of $\SL(V)$ acting with weights \[\lambda_1 \leq \cdots \leq \lambda_n\] on a basis $\{x_1,\dots, x_n\}$ of $V$. 
By Proposition \ref{P:assoc}, the set of the initial monomials of $I_{n(d-1)}$ contains all monomials of degree $n(d-1)$ with the
exception of a single monomial $m_0$ such that $m_0 \geq_{\lambda} x_1^{d-1}\cdots x_n^{d-1}$.
Then the sum of the $\lambda$-weights of all initial monomials of $I_{n(d-1)}$
satisfies\footnote{Clearly, the sum of $\lambda$-weights of all degree $n(d-1)$ monomials is $0$.}
\[
\text{the sum} = 0-w_{\lambda}(m_0) \leq - w_{\lambda} (x_1^{d-1}\cdots x_n^{d-1})=0.
\]
It follows by Proposition \ref{HM-grass} that $I_{n(d-1)}$ is semistable with respect to $\lambda$.
\end{proof}

\begin{proof}[Proof of Proposition \ref{P:assoc}]

Since $S/I$ is a graded local Artinian Gorenstein $\CC$-algebra with socle in degree $n(d-1)$, we have that
$I_{n(d-1)}$ has codimension $1$ in $\Sym^{n(d-1)}V$. In particular, the set of the initial monomials of $I_{n(d-1)}$ with respect to $\lambda$
is the set of all degree $n(d-1)$ monomials with the exception of exactly one monomial $m_0$. 
It remains to prove that 
\[
m_0 \geq_{\lambda} x_1^{d-1}\cdots x_n^{d-1}.
\]

We argue by contradiction. Suppose $m_0 <_{\lambda} x_1^{d-1}\cdots x_n^{d-1}$. Then all monomials greater than or
equal to $x_1^{d-1}\cdots x_n^{d-1}$ with respect to $<_{\lambda}$ are the initial monomials of $I_{n(d-1)}$. 
It follows that every monomial greater than or equal to $x_1^{d-1}\cdots x_n^{d-1}$ with respect to $<_{\lambda}$
actually belongs to $I_{n(d-1)}$. 

Suppose now that $m_1=x_1^{d_1}\cdots x_n^{d_n}$ is \emph{some} monomial of degree $n(d-1)$ that \emph{does not} belong to $I_{n(d-1)}$
(for example, we can take $m_1=m_0$). Then $m_1 <_{\lambda} x_1^{d-1}\cdots x_n^{d-1}$.
\begin{claim} 
There must exist an index $i\in \{1,\dots, n-1\}$ such that 
\[
d_1+\cdots+d_i>i(d-1).
\]
\end{claim}
\begin{proof}[Proof of claim]
This follows by combining the assumptions $\lambda_1\leq \cdots \leq \lambda_n$ and
$m_1 <_{\lambda} x_1^{d-1}\cdots x_n^{d-1}$. 
\end{proof}

Take $i\in \{1,\dots, n-1\}$ such that $d_1+\cdots+d_i>i(d-1)$. 
For every $1\leq k \leq n$, set \[
h_k(x_1,\dots,x_i):=g_k(x_1,\dots,x_i, 0,\dots, 0)=g_k \mod (x_{i+1},\dots, x_n).
\]

Next, we recall a well-known Bertini-type result: 
\begin{lemma}\label{L:bertini}
Suppose $L \subset \Sym^d W$ is a subspace with no base locus in $\bP W^{\vee}$ and 
$\dim L \geq \dim W$. If $\dim L > \dim W$, then
a general hyperplane in $L$ defines a base-point-free linear system as well. In particular, a general
$\bigl(\dim W \bigr)$-dimensional linear subspace of $L$ defines a base-point-free linear system,
hence is generated by a regular sequence of degree $d$ forms.
\end{lemma}
\begin{proof} Set $r=\dim L$ and $n=\dim W$. Assume $r>n$.
Consider the incidence correspondence \[I=\{(p, D) \mid p\in D, D\in L\}\subset \bP^{n-1}\times L.\] Since $L$ is base-point-free, 
the projection $I \to \bP^{n-1}$ is a $\CC^{r-1}$-bundle. As long as the hyperplane $H \subset L$ is not one of these
fibers, we are done. The claim follows by dimension count: $n-1=\dim \bP^{n-1} < r-1 =\dim \Grass(r-1, L)$.
\end{proof}

Set $W=\lspan \langle x_1,\dots,x_i\rangle$.  
Applying Lemma \ref{L:bertini} to the base-point-free linear system $\lspan \langle  h_1,\dots,h_n\rangle$ in $\Sym^d W$,
we find that there exists a regular sequence $\{f_1,\dots, f_i\}$ of elements in $\Sym^d W$ such that $(f_1,\dots, f_i)\subset (h_1,\dots,h_n)$.
Then $\left(\Sym W\right)/(f_1,\dots, f_i)$ is a graded local Artinian Gorenstein algebra with socle in degree $i(d-1)$. It follows that
every monomial in variables $x_1,\dots,x_i$ of degree greater than $i(d-1)$ lies in $(f_1,\dots, f_i)\subset (h_1,\dots,h_n)$.

Going back to the monomial $m_1$ and recalling the assumption \[D:=d_1+\cdots+d_i>i(d-1),\] we have
\[
x_1^{d_1}\cdots x_i^{d_i} \in I_D \mod{(x_{i+1}, \dots, x_n)}.
\]
Since $m_1=x_1^{d_1}\cdots x_i^{d_i} x_{i+1}^{d_{i+1}}\cdots x_n^{d_n} \notin I_{n(d-1)}$, there must be a monomial
in \[(x_{i+1}, \dots, x_n)(x_1,\dots,x_n)^{D-1}\] such that its product with $x_{i+1}^{d_{i+1}}\cdots x_n^{d_n}$ is not in $I_{n(d-1)}$.
We have arrived at the conclusion that there exists a monomial 
\[
m_2=x_1^{d'_1}\dots x_{i}^{d'_i} x_{i+1}^{d'_{i+1}}\cdots x_n^{d'_n}
\]
 that \emph{does not} belong to $I_{n(d-1)}$ and such that 
 \begin{equation}\label{E:exponents}
 d'_{i+1} \geq d_{i+1}, \ d'_{i+2} \geq d_{i+2},  \ \dots, \ d'_{n} \geq d_{n},
 \end{equation}
and such that at least one of the above inequalities is strict. 

Repeating this process, we obtain an infinite sequence of monomials
$m_1, m_2, \dots$ such that the exponents of $x_1,\dots, x_n$ for any two successive monomials
satisfy the inequalities \eqref{E:exponents} (possibly with a different $i$ each time), with at least one inequality strict. This is however absurd:
Clearly, $d_{n}$ has to stabilize, which forces $i$ to be less than $n$ from then on, which forces $d_{n-1}$ to stabilize, etc. 
\end{proof}

\subsection*{Acknowledgements.} I am grateful to Jarod Alper for introducing me to the problem 
of GIT stability of Hilbert points of Milnor algebras. 
I had many illuminating discussions on the subject with Jarod Alper and Alexander Isaev. 
I would also like to thank Alexander Isaev and Ian Morrison for comments on earlier drafts.  
This research was partially supported by the NSF grant DMS-1259226 and a Sloan Research Fellowship. 
The first part of the paper was completed at the Columbia University in June 2015. 
The second part was completed at the Max Planck Institute for Mathematics in Bonn in October 2015. 
I thank these institutions for their hospitality.

\bibliography{NSF_bib}{}

\def\cprime{$'$} \def\cprime{$'$} \def\cprime{$'$}
\begin{thebibliography}{MFK94}

\bibitem[AI14]{alper-isaev-assoc}
Jarod Alper and Alexander Isaev.
\newblock Associated forms in classical invariant theory and their applications
  to hypersurface singularities.
\newblock {\em Math. Ann.}, 360(3-4):799--823, 2014.

\bibitem[AI15]{alper-isaev-binary}
Jarod {Alper} and Alexander {Isaev}.
\newblock {Associated forms and hypersurface singularities: The binary case},
  2015.
\newblock To appear in \emph{Journal f\"ur die Reine und Angewandte
  Mathematik}, preprint {\tt arXiv:1407.6838v3}.

\bibitem[AM69]{AM}
M.~F. Atiyah and I.~G. Macdonald.
\newblock {\em Introduction to commutative algebra}.
\newblock Addison-Wesley Publishing Co., Reading, Mass.-London-Don Mills, Ont.,
  1969.

\bibitem[Don83]{donagi}
Ron Donagi.
\newblock Generic {T}orelli for projective hypersurfaces.
\newblock {\em Compositio Math.}, 50(2-3):325--353, 1983.

\bibitem[GKZ94]{GKZ}
I.~M. Gel{\cprime}fand, M.~M. Kapranov, and A.~V. Zelevinsky.
\newblock {\em Discriminants, resultants, and multidimensional determinants}.
\newblock Mathematics: Theory \& Applications. Birkh\"auser Boston, Inc.,
  Boston, MA, 1994.

\bibitem[Kem78]{kempf}
George~R. Kempf.
\newblock Instability in invariant theory.
\newblock {\em Ann. of Math. (2)}, 108(2):299--316, 1978.

\bibitem[Lun75]{luna-adherences}
Domingo Luna.
\newblock Adh\'erences d'orbite et invariants.
\newblock {\em Invent. Math.}, 29(3):231--238, 1975.

\bibitem[MFK94]{GIT}
David Mumford, John Fogarty, and Frances Kirwan.
\newblock {\em Geometric invariant theory}, volume~34 of {\em Ergebnisse der
  Mathematik und ihrer Grenzgebiete (2) [Results in Mathematics and Related
  Areas (2)]}.
\newblock Springer-Verlag, Berlin, third edition, 1994.

\end{thebibliography}
\bibliographystyle{alpha}
\end{document}